\documentclass[a4paper]{amsart}
\usepackage{amssymb}
\usepackage[arrow, matrix, curve]{xy}
\usepackage{comment}
\usepackage{stmaryrd}

\newcommand{\Edges}{\text{Edges}}

\def\to{\rightarrow}

\newtheorem{theorem}{Theorem}[section]
\newtheorem{lemma}[theorem]{Lemma}
\newtheorem{prop}[theorem]{Proposition}
\newtheorem{cor}[theorem]{Corollary}

\theoremstyle{definition}
\newtheorem{definition}[theorem]{Definition}

\theoremstyle{remark}
\newtheorem{remark}[theorem]{Remark}

\numberwithin{equation}{section}

\begin{document}

\title[Note on the tensor product of dendroidal sets]{Note on the tensor product of dendroidal sets}


\author[D.-C. Cisinski]{Denis-Charles Cisinski}

\author[I. Moerdijk]{Ieke Moerdijk}




%

\maketitle

In this note, we provide a detailed proof of Lemma 1.11 in \cite{erratum},
rephrased here as the combination of Corollaries \ref{cor:1}
and \ref{cor:2}. As we explain in \cite{erratum}, theses two facts imply that
the operadic model structure on the category of dendrodial sets is enriched
in simplicial sets with the Joyal model structure, respectively
that its restriction to open dendroidal sets is a symmetric monoidal
model category. We emphasize that the corrections required for our paper
\cite{dSet model hom op} and described in \cite{erratum} only need
the first fact, and hence only need Corollary \ref{cor:1} below.


\section{Tensoring with simplices}

Let us show the following statement.

\begin{prop}\label{prop:1.1}
For any $n\geq 0$ and any object $T$ of $\Omega$, the map 
	\begin{equation*}
	(\partial\Delta[n] \otimes\Omega[T])
	\cup (\Delta[n] \otimes \partial\Omega[T]) \rightarrow \Delta[n] \otimes \Omega[T]
	\end{equation*}
is a normal monomorphism. 
\end{prop}

The map here is from the pushout over $\partial\Delta[n]\otimes \partial \Omega[T]$.
Note that, if $S$ is any tree, $\Omega[S]\otimes\Omega[T]$ is the dendroidal nerve of a
$\Sigma$-cofibrant operad, and is therefore a normal dendroidal set.
As, in the category of dendroidal sets,
any monomorphism with normal target is a normal monomorphism,
it is only the property of being a monomorphism which
needs to be checked here.

In what follows, we will abreviate the notations by writing
$T=\Omega[T]$ and $[n]=\Delta[n]$.

\begin{proof}
This property amounts to showing that both maps
$\partial[n]\otimes T \rightarrow [n] \otimes T$ and $[n]\otimes \partial T \rightarrow [n]\otimes T$ are mono, and their pullback is contained in $\partial[n]\otimes\partial T$; in other words, using that the canonical maps of the form
$\partial_i[n]\otimes T\rightarrow [n]\otimes T$
and $[n]\otimes\partial_x T\rightarrow[n]\otimes T$ are
easily seen to be monomorphisms, we have to prove
that one has the following
identifications:

\begin{enumerate}
\item[(A1)] $\partial_i[n]\otimes T \cap \partial_j[n]\otimes T = (\partial_i[n] \cap \partial_j[n])\otimes T$;
\item[(A2)] $[n] \otimes \partial_x T \cap [n]\otimes \partial_y T = [n]\otimes (\partial_x T \cap \partial_y T)$;
\item[(A3)] $\partial_i[n]\otimes T \cap [n]\otimes \partial_xT = \partial_{i}[n]\otimes \partial_x T$.
\end{enumerate}

The inclusions from right to left are obvious and the inclusions form left to right will
follow from a suitable way of encoding dendrices of $[n]\otimes T$, as expressed in
Lemma \ref{lemma:1} below.
\end{proof}

By definition, a dendrex of $[n]\otimes T$ is a map $S\rightarrow A$ where $A$ is a shuffle of $[n]$ and $T$. Since $A\rightarrowtail [n]\otimes T$ and $S\rightarrow A$ factors as $S\twoheadrightarrow S' \rightarrowtail A$, every such dendrex factors as a degeneracy followed by a mono. So we will look at monos 
	\begin{equation*}
	S \rightarrowtail [n] \otimes T,
	\end{equation*}
which are then faces of shuffles of $[n]$ and $T$. We will use the natural partial ordering on the edges of a tree where the root is minimal and the leaves are maximal. 

For an integer $n\geq 0$, a \emph{height function} on a tree $S$ is \emph{multi-valued} function 
	\begin{equation*}
	h\colon \Edges(S) \rightarrow \{0,1,\dots,n\}
	\end{equation*}
with the property that if $e\leq e'$ and $i\in h(e), i'\in h(e')$, then $i\leq i'$. (We might call such an $h$ \emph{monotone}). 

For example, for $n=5$ and 
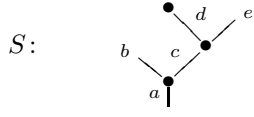
\begin{equation*}
\xymatrix@R=3pt@C=8pt{
&&&*{ \bullet}&& _{e}&&&& \\
S\colon &  & _b && *{\bullet} \ar@{-}[ul]_{d}  \ar@{-}[ur] &&& \\
&&&*{\bullet}  \ar@{-}[ul] \ar@{-}[ur]^c  &&&&\\
&&&*=0{}\ar@{-}[u]^a &&&&
}
\end{equation*}  

the picture 

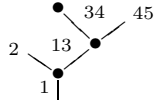
\begin{equation*}
\xymatrix@R=3pt@C=8pt{
&&&*{ \bullet}&& _{45}&&&& \\
&& _{2} && *{\bullet} \ar@{-}[ul]_{34}  \ar@{-}[ur] &&& \\
&&&*{\bullet}  \ar@{-}[ul] \ar@{-}[ur]^{13}  &&&&\\
&&&*=0{}\ar@{-}[u]^{1} &&&&
}
\end{equation*}  

indicates the height function $h(a)=\{1\}$, $h(c)=\{1,3\}$, etc. 

\begin{lemma}\label{lemma:1}
There is a bijective correspondence between monos $S' \rightarrowtail [n]\otimes T$ and pairs ($u\colon S\rightarrowtail T, h)$ where $u\colon S\rightarrow T$ is a face of $T$ and $h$ is a height function on $S$. 
This correspondence has the following properties.
\begin{enumerate}
\item[(1)] For any $i\in \{0,1,\dots, n\}$ the map  $S' \rightarrowtail [n]\otimes T$ factors through  $\partial_i[n]\otimes T$ if and only if $i$ does not occur in the image
of $h$. For any $I \subseteq \{0,1,\dots, n\}$ the map  $S' \rightarrowtail [n]\otimes T$ factors through  $\bigcap_{i\in I}\partial_i[n]\otimes T$ if and only if the image of $h$ does not contain any element of $I$. 
\item[(2)] For any face $F\rightarrowtail T$, the map  $S' \rightarrowtail [n]\otimes T$ factors through $[n]\otimes F$ if and only if $u\colon S \rightarrowtail T$ factors through $F$. 
\end{enumerate}
\end{lemma}

\emph{Example.} The tree $S$ with height function $h$ pictured in (1) and (2) above corresponds to the tree $S'$ : 
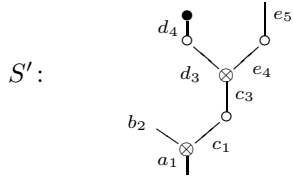
\begin{equation*}
\xymatrix@R=3pt@C=8pt{
&&& 			 && 			  &&&& \\
&&& *{ \bullet} 			 && 			  &&&& \\
&&& *{\circ} \ar@{-}[u]^{d_4}&& *{\circ} \ar@{-}[uu]_{e_5} &&&& \\
S'\colon &&&& *{ _\otimes} \ar@{-}[ul]^{d_3}  \ar@{-}[ur]_{e_4}		&&&&& \\
 &  & _{b_2} && *{\circ} \ar@{-}[u]_{c_3}  &&& \\
&&&*{_\otimes}  \ar@{-}[ul] \ar@{-}[ur]_{c_1}   &&&&\\
&&&*=0{}\ar@{-}[u]^{a_1} &&&&
}
\end{equation*}  
which is a face of the ``complete'' shuffle $A$ of $[n]\otimes S$:
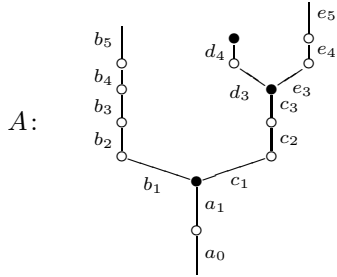
\begin{equation*} 
\xymatrix@R=3pt@C=8pt{
&&&& 			 && 			  &&&& \\
&&&& 			 && 			  &&&& \\
&&&&& *{ \bullet} 			 && 	*{\circ} \ar@{-}[uu]_{e_5}		  &&& \\
&&*{\circ} \ar@{-}[uu]^{b_5} &&& *{\circ} \ar@{-}[u]^{d_4}&& *{\circ} \ar@{-}[u]_{e_4} &&& \\
 && *{\circ} \ar@{-}[u]^{b_4} &&&& *{ \bullet} \ar@{-}[ul]^{d_3}  \ar@{-}[ur]_{e_3}		&&&& \\
 A\colon &  & *{\circ} \ar@{-}[u]^{b_3} &&&& *{\circ} \ar@{-}[u]_{c_3}  && \\
 &  & *{\circ} \ar@{-}[u]^{b_2} &&&& *{\circ} \ar@{-}[u]_{c_2}  && \\
&&&&*{\bullet}  \ar@{-}[ull]^{b_1} \ar@{-}[urr]_{c_1}   &&&&\\
&&&& 			 && 			  &&&& \\
&&&&*{\circ} \ar@{-}[uu]_{a_1} &&&& \\
&&&& 			 && 			  &&&& \\
&&&&\ar@{-}[uu]_{a_0} &&&& 
}
\end{equation*}

\begin{proof}[Proof of lemma \ref{lemma:1}]
Given a complete shuffle $A\subseteq [n]\otimes T$ and a face $S' \rightarrowtail A$, consider the map 
	\begin{equation*}
\Edges(S') \rightarrow \Edges(A) \subseteq \Edges [n] \times \Edges (T) =\{0,\dots, n\}\times \Edges(T)
	\end{equation*}

Write $h': \Edges(S')\rightarrow \{0,\dots, n\}$ and $u' \colon \Edges(S') \rightarrow \Edges (T)$ for the two factors of this map. The function $u'$ in fact defines a map $S' \to T$ in $\Omega$, which we factor as 
	\begin{equation*}
	u \circ \sigma\colon S' \twoheadrightarrow S \rightarrow T.
	\end{equation*}
Let $h=h'\circ \sigma^{-1} :\Edges(S)\rightarrow \{0,1,\dots,n\}$. Then $h$ is a height function on $S$. 

In the above picture, $S$ is obtained from $S'$ by deleting the white vertices
as well as the numbers occuring as subscripts of the letters,
and $h$ is obtained by associating to an edge of $S$ the set of all numbers on corresponding edges of $S'$.

Conversely, suppose $u\colon S\rightarrowtail T$ and $h$ is a height function on $S$. We will construct a shuffle $A$, in fact of $[n]\otimes S \rightarrowtail [n]\otimes T$, and a map $\sigma \colon S' \rightarrow S$. We start by subdividing each edge $e$ in $S$ into a number of ``shorter'' edges $(e,i)$ for $i\in h(e)$. 
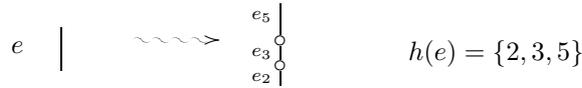
\begin{equation*}
\xymatrix@R=3pt@C=8pt{
&&& \\
&&& \\
&e && \\
&&\ar@{-}[uu] &
}
\xymatrix@R=3pt@C=8pt{ \\ \\ \ar@{~>}[rrr] & && } 
\xymatrix@R=3pt@C=8pt{
&&&& \\
&&&& \\
&*{\circ} \ar@{-}[uu]^{e_5}&&& \\
&*{\circ} \ar@{-}[u]^{e_3}&&&\\
&\ar@{-}[u]^{e_2} &&&
}
\xymatrix@R=3pt@C=8pt{ \\ \\ h(e)=\{2,3,5\} }
\end{equation*}

The resulting tree $S'$ has the same shape as $S$ (it permits a degeneracy $\sigma\colon S' \twoheadrightarrow S$), and has non-decreasing numbers along each path going up from the root. To define the map $S' \rightarrow [n]\otimes S\rightarrow [n]\otimes T$, we construct a complete shuffle of $[n]\otimes S$ and a face map $S' \rightarrowtail A$. (This shuffle will not be unique, but the composition $S'\rightarrow A \rightarrow [n]\otimes T$ is completely determined by $S\rightarrowtail T$ and $h$.) The complete shuffle will be a further subdivision of $S'$. Consider an edge $e$ of $S$ and its subdivision in $S'$ into edges $(e,i), i\in h(e)$; say $(e,i_1),\dots,(e, i_k)$ where $h(e)=\{i_1,\dots i_k\}$. 
\begin{itemize}
\item If $e$ is the root edge, we subdivide it further into $(e,i)$ for $0\leq i \leq i_k$.
\item Otherwise $e$ is the input edge of a vertex $v$ with output edge $d$, say. Let $j=\max h(d) \leq i_1$. If $e$ is an inner edge, we subdivide it further into $(e,i)$ where $j\leq i\leq i_k$. And if $e$ is a leaf, into $(e,i)$ with $j\leq i\leq n$. 
\end{itemize}
The resulting tree is a complete shuffle, with an obvious face map $S' \rightarrow A$ defined by deleting the newly added faces in $A$ which are not in $S'$. Properties $(1)$ and $(2)$ are clear from the construction. 
\end{proof}
%

\begin{remark}
The shuffle $A$ is uniquely determined by $S$ and $h$, but is not the only shuffle through which the corresponding map $S' \rightarrow [n]\otimes S$ factors. For example, 
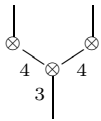
\begin{equation*}
\xymatrix@R=3pt@C=8pt{
&&&&& &&&& \\
&&&&& &&&& \\
&& *{_\otimes} \ar@{-}[uu] && *{_\otimes} \ar@{-}[uu]&&&\\
&&&*{_\otimes}  \ar@{-}[ul]^4 \ar@{-}[ur]_{4}  &&&&\\
&&&&& &&&& \\
&&&*=0{}\ar@{-}[uu]^{3} &&&&
}
\end{equation*}  
\noindent factors through the left hand tree 
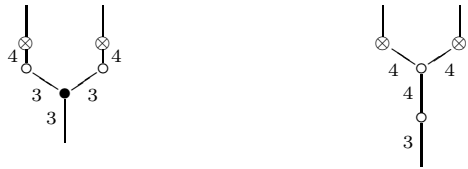
\begin{equation*}
\xymatrix@R=3pt@C=8pt{
&&&&& &&&& \\
&&&&& &&&& \\
&& *{_\otimes} \ar@{-}[uu] && *{_\otimes} \ar@{-}[uu]&&&\\
&& *{\circ} \ar@{-}[u]^4 && *{\circ} \ar@{-}[u]_4&&&\\
&&&*{\bullet}  \ar@{-}[ul]^3 \ar@{-}[ur]_{3}  &&&&\\
&&&&& &&&& \\
&&&*=0{}\ar@{-}[uu]^{3} &&&&
}
\xymatrix@R=3pt@C=8pt{
&&&&& &&&& \\
&&&&& &&&& \\
&& *{_\otimes} \ar@{-}[uu] && *{_\otimes} \ar@{-}[uu]&&&\\
&&&*{\circ}  \ar@{-}[ul]^4 \ar@{-}[ur]_{4}  &&&&\\
&&&&& &&&& \\
&&&*{\circ}  \ar@{-}[uu]^4  &&&&\\
&&&&& &&&& \\
&&&*=0{}\ar@{-}[uu]^{3} &&&&
}
\end{equation*}
\noindent as constructed, but also through the one on the right by the
Boardman-Vogt relation. 
\end{remark}

\begin{remark}
The argument breaks down if we replace $[n]$ by its closure
$\textrm{cl}[n]$ (obtained by adding a bald vertex on top of this linear tree),
because, for $S'\rightarrowtail \textrm{cl}[n]\otimes T$, the corresponding tree $S$ (with height function) need not map to $T$ because the valence may have gone down, as in 
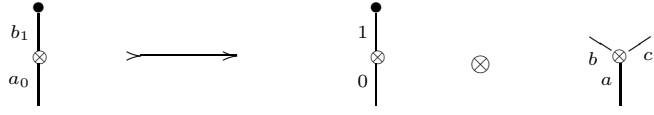
\begin{equation*}
\xymatrix@R=3pt@C=8pt{
&&& *{\bullet} && \\
&&&&&  \\
&&&*{\, _\otimes}  \ar@{-}[uu]^{b_1}  \\
&&&&&  \\
&&&*=0{}\ar@{-}[uu]^{a_0} 
}
\xymatrix@R=3pt@C=8pt{ \\ \\ \ar@{>->}[rrr] & && }
\xymatrix@R=3pt@C=8pt{
&&& *{\bullet} && \\
&&&&&  \\
&&&*{\, _\otimes}  \ar@{-}[uu]^{1}  \\
&&&&&  \\
&&&*=0{}\ar@{-}[uu]^{0} 
}
\xymatrix@R=3pt@C=8pt{ \\ \\ \otimes }
\xymatrix@R=3pt@C=8pt{
&&&&& &&&& \\
&&&&& &&&& \\
&&&*{_\otimes}  \ar@{-}[ul]^b \ar@{-}[ur]_{c}  &&&&\\
&&&&& &&&& \\
&&&*=0{}\ar@{-}[uu]^{a} &&&&
}
\end{equation*}

\end{remark}

\begin{cor}\label{cor:1}
Let $A \rightarrowtail X$ be a normal monomorphism and $B \rightarrowtail Y$ a monomorphism of simplicial sets. Then 
	\begin{equation*}
A\otimes i_!Y \cup_{A \otimes i_!B} X\otimes  i_!B \rightarrow X\otimes i_!Y
	\end{equation*}
is again a normal monomorphism.
\end{cor}

\section{Tensoring open trees}
\subsection{Faces of a tree}

Let $T$ be a tree. Recall that a \emph{face} of $T$ is a subtree of the form 
\begin{itemize}
\item $\partial_x(T)$, where $x$ is an internal edge (``internal face'')
\item $\partial_v(T)$, where $v$ is a top vertex (``external/top face'')
\item $\partial_r(T)$, where $x$ is an internal edge (``external/root face'')
\end{itemize}
The root face is defined only if there is just one internal edge attached to the root. In addition, by definition, for a corolla $C$ any edge of $C$ is a face of $C$. 
(For a composition $S \rightarrowtail T$ of face inclusions,
we sometimes refer to $S$ as a face of higher codimension; or --- for emphasis --- to the faces of codimension 1 listed above as ``elementary'' faces.)

For a leaf $y$ of a tree $T$, attached to a vertex $v$, we define
	\begin{equation*}
	D_y(T)
	\end{equation*}
to be the sum of the connected components of the graph (a forest) obtained from $T$ by deleting the vertex $v$ and the edge $y$. To be more explicit, for a given edge $z$ in $T$ we will write $T/z$ for a subtree with root edge $z$. 
Similarly, for a vertex $v$, we will write $v/T$ for the connected component of $T-v$ containing the root edge. 
Then 
	\begin{equation*}
	D_y(T)= T/y_1 + \dots + T/y_n +v/T
	\end{equation*}
where $y=y_0$, $y_1$, ..., $y_n$ are the input edges of $v$ (one also says $y_1,\dots, y_n$ are ``siblings'' of y). 


Similarly, if $x_1, \dots, x_n$ are the inputs of the root vertex $r$, and $y$ is the root edge, we define 
	\begin{equation*}
	D_r(T)= T/x_1 + \dots + T/x_n=D_y(T).
	\end{equation*}
Finally, if $y$ is an internal edge, we will also write 
	\begin{equation*}
	D_y(T)=\partial_y(T).
	\end{equation*}

\subsection{Intersection of faces}

Let us list the intersection of various types of pairs of faces: 
\begin{enumerate}
\item[(a)] For two external top faces $\partial_v T$ and $\partial_w T$ we have 
	\begin{equation*}
	\partial_v\partial_w T = \partial_v T \cap \partial_w T = \partial_w \partial_v T.
	\end{equation*}
\item[(b)] For two internal faces $\partial_x T$ and $\partial_y T$ we similarly have 
	\begin{equation*}
	\partial_x\partial_y T = \partial_x T \cap \partial_y T = \partial_y \partial_x T.
	\end{equation*}
\item[(c)] For an internal edge $x$ and a top vertex $v$ (or the root vertex $r$ if $\partial_rT$ is defined), if $x$ is not attached to $v$ then again  
	\begin{equation*}
	\partial_x\partial_v T = \partial_x T \cap \partial_v T = \partial_v \partial_x T.
	\end{equation*}
\item[(d)] For the internal edge $x$ attached to a top vertex $v$ we have 
	\begin{equation*}
	\partial_x T \cap \partial_v T = D_x (\partial_v T).
	\end{equation*}
\item[(e)] Similarly, for the root face $\partial_r T$ and the internal edge $y$ attached to $r$, we have 
	\begin{equation*}
	\partial_y T \cap \partial_r T = D_y (\partial_r T).
	\end{equation*}
\end{enumerate}

\begin{definition}
A tree $T$ is called \emph{open} if it has no bald vertices (no stumps, no nullary vertices). Notice that the full subcategory of open trees forms a sieve in $\Omega$. In other words, if $S \rightarrow T$ is a morphism in $\Omega$ and $T$ is open then so is $S$. 
\end{definition}

\subsection{The shuffle lemma for open trees}

Recall that for two trees $S$ and $T$, their tensor product is a union of representable presheaves $A\subseteq S\otimes T$ where $A$ is a tree obtained as a \emph{shuffle} of $S$ and $T$ --- we also say ``a shuffle of $S\otimes T$''. 

\begin{lemma}
Let $S$ and $T$ be trees, and assume $S$ is \emph{open}. Let $y$ be an edge of $T$, and $A\subseteq S\otimes T$ a shuffle of $S\otimes T$. If $F$ is a face of $A$ which does not contain $y$, then $F$ is contained in $S\otimes D_yT$ (i.e. in $S\otimes R$ for a connected component $R$ of $D_yT$). 
\end{lemma}

\begin{proof}
Let $F\subseteq A\subseteq S\otimes T$ be a face of shuffle $A$ not containing the colour $y$. We distinguish three cases. 
\begin{enumerate}
\item[(i)] The edge $y$ is internal in $T$. 
Consider the subtrees 
	\begin{equation*}
	R_1, \dots, R_k 
	\end{equation*}
of $A$ given by the edge with colour $y$. These are all internal edges of $A$, and $F$ is contained in the face $G$ (of higher codimension) of $A$ obtained by contracting all these edges. By the BV--relation, $G$ is also a face of the shuffle $B$ obtained from $A$ by pushing up the $S$-vertices in these subtrees $R_1,\dots, R_k$, so that no two occurrences of $y$ in $B$ are connected. Here is a picture illustrating this for small trees $S$ and $T$:
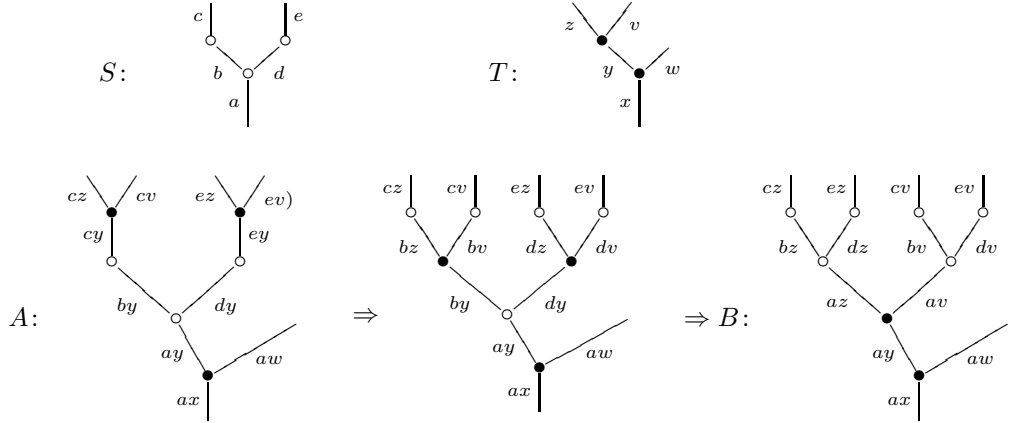
\begin{equation*}
\xymatrix@R=3pt@C=8pt{
&&&&& &&&& \\
&&&&& &&&& \\
&& *{\circ} \ar@{-}[uu]^c && *{\circ} \ar@{-}[uu]_e&&&\\
S\colon &&&*{\circ}  \ar@{-}[ul]^b \ar@{-}[ur]_{d}  &&&&\\
&&&&& &&&& \\
&&&\ar@{-}[uu]^{a} &&&&
}
\xymatrix@R=3pt@C=8pt{
&&&&& &&&& \\
&& &&&&&\\
&& *{\bullet} \ar@{-}[uul]^z \ar@{-}[uur]_v &&&&&\\
T\colon&&&*{\bullet}  \ar@{-}[ul]^y \ar@{-}[ur]_{w}  &&&&\\
&&&&& &&&& \\
&&&\ar@{-}[uu]^{x} &&&&
}
\end{equation*}
\begin{equation*}
\xymatrix@R=3pt@C=6pt{
&&&&&&&&& \\
&&&&&&&&& \\
&&*{\bullet} \ar@{-}[uul]^{cz} \ar@{-}[uur]_{cv} &&&&*{\bullet} \ar@{-}[uul]^{ez} \ar@{-}[uur]_{ev)} \\
 \\
&&*{\circ} \ar@{-}[uu]^{cy}&&&& *{\circ} \ar@{-}[uu]_{ey}\\
 \\
A\colon &&&& *{\circ} \ar@{-}[uull]^{by} \ar@{-}[uurr]_{dy} &&&&\\
\\
&&&&&*{\bullet}  \ar@{-}[uul]^{ay} \ar@{-}[uurrr]_{aw}  \\
\\
&&&&&\ar@{-}[uu]^{ax} 
}
\xymatrix@R=3pt@C=6pt{
&&&&&&&&& \\
&&&&&&&&& \\
&*{\circ} \ar@{-}[uu]^{cz} &&*{\circ} \ar@{-}[uu]^{cv}&&*{\circ} \ar@{-}[uu]^{ez}&&*{\circ} \ar@{-}[uu]^{ev} && \\
 \\
&&*{\bullet} \ar@{-}[uul]^{bz} \ar@{-}[uur]_{bv} &&&& *{\bullet} \ar@{-}[uul]^{dz} \ar@{-}[uur]_{dv} \\
 \\
\Rightarrow &&&& *{\circ} \ar@{-}[uull]^{by} \ar@{-}[uurr]_{dy} &&&&\\
\\
&&&&&*{\bullet}  \ar@{-}[uul]^{ay} \ar@{-}[uurrr]_{aw}  \\
\\
&&&&&\ar@{-}[uu]^{ax} 
}
\xymatrix@R=3pt@C=6pt{
&&&&&&&&& \\
&&&&&&&&& \\
&*{\circ} \ar@{-}[uu]^{cz} &&*{\circ} \ar@{-}[uu]^{ez}&&*{\circ} \ar@{-}[uu]^{cv}&&*{\circ} \ar@{-}[uu]^{ev} && \\
 \\
&&*{\circ} \ar@{-}[uul]^{bz} \ar@{-}[uur]_{dz} &&&& *{\circ} \ar@{-}[uul]^{bv} \ar@{-}[uur]_{dv} \\
 \\
\Rightarrow B\colon&&&& *{\bullet} \ar@{-}[uull]^{az} \ar@{-}[uurr]_{av} &&&&\\
\\
&&&&&*{\bullet}  \ar@{-}[uul]^{ay} \ar@{-}[uurrr]_{aw}  \\
\\
&&&&&\ar@{-}[uu]^{ax} 
}
\end{equation*}
Contracting these isolated internal edges $y$ in $B$ defines a shuffle of $S\otimes\partial_yT$, isomorphic to $G$ (as subobject of $S\otimes T$) and hence containing $F$. 
\item[(ii)] The edge $y$ is a leaf of $T$. 
Write, as before, 
	\begin{equation*}
	D_y(T)= T/y_1 + \dots + T/y_n +v/T
	\end{equation*}
for the connected components of $T-\{y, v\}$. Since $S$ is open the colour $y$ also occurs on the leaves of $A$ (in fact, as many times as the number of leaves of $S$). Since $y$ does not occur in the face $F$, this face must be a face (of higher codimension) contained in one of the connected components of $A$ obtained by deleting all these edges with colour $y$ from $A$ (and the vertices attached to these edges, of course). But these connected components are shuffles of $S' \otimes T/y_i$ of of $S' \otimes v/T$ for subtrees $S'$ of $S$. 
(In fact, one knows exactly which subtrees: if 
	\begin{equation*}
	(b_i, y), \quad i=1,\dots ,k
	\end{equation*}
is a list of all lowest occurrences of the colour $y$, then the shuffles are of $S/b_i \otimes T/y_j$, and of $R\otimes v/T$ where $R\subseteq S$ is the subtree of $S$ with the same root and leaves $b_1,\dots, b_k$. )
\item[(iii)] The edge $y$ is the root edge of $T$. 
In this case the colour $y$ also occurs on the root of the shuffle $A$. In fact, the edges coloured $y$ form a subtree of $A$ isomorphic to a subtree $S'$ of $S$ containing the root of $S$ (an ``initial segment''), with leaves $b_1,\dots, b_k$ say. 
Then $F$ must be a face of $S/b_i \otimes T/y_j \subseteq S/b_i \otimes D_rT$ for some input edge $y_j$ of the root $r$ of $T$. Hence $F\subseteq S\otimes D_y(T)=S\otimes D_r(T)$. 
\end{enumerate}
\end{proof}

\begin{remark}
The lemma fails if $S$ is not assumed to be open. A minimal example is 
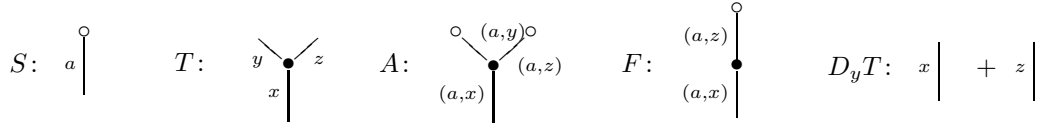
\begin{equation*}
\xymatrix@R=3pt@C=8pt{
&&&  \\
& *{\circ} && \\
S\colon&&&  \\
& \ar@{-}[uu]^{a}  
}
\xymatrix@R=3pt@C=8pt{
&&& \\
&&&  \\
T\colon&&*{\bullet}  \ar@{-}[ul]^y \ar@{-}[ur]_{z}  \\
&&&&  \\
&&*=0{}\ar@{-}[uu]^{x} 
}
\xymatrix@R=3pt@C=8pt{
&&&&&  \\
&*{\circ}&& *{\circ}&  \\
A\colon&&*{\bullet}  \ar@{-}[ul]_{(a,y)} \ar@{-}[ur]_{(a,z)}  \\
&&&&&  \\
&&*=0{}\ar@{-}[uu]^{(a,x)} 
}
\xymatrix@R=3pt@C=8pt{
&& *{\circ} && \\
&&&&  \\
F\colon&&*{\bullet}  \ar@{-}[uu]^{(a,z)}  \\
&&&&  \\
&&\ar@{-}[uu]^{(a,x)} 
}
\xymatrix@R=3pt@C=8pt{
&&&&&  \\ &&&&&  \\ D_yT\colon &&+&&  \\
& \ar@{-}[uu]^{x}  &  & \ar@{-}[uu]^{z} 
}
\end{equation*}
\end{remark}

\begin{prop} \label{prop}
Let $S$ and $T$ be trees, and let $x,y,z$ be inner edges or top or root vertices
(when the root face is defined), with $x$ in $S$ and $y,z$ in $T$. 
\begin{itemize}
\item[(i)] If $S$ is open then 
	\begin{equation*}
	S\otimes \partial_y T \cap S\otimes \partial_z T= S\otimes (\partial_y T \cap \partial_z T).
	\end{equation*}
\item[(ii)] If $S$ and $T$ are both open then 
	\begin{equation*}
	\partial_x S\otimes  T \cap S\otimes \partial_y T= \partial_xS\otimes \partial_y T.
	\end{equation*}
\end{itemize}
\end{prop}

\begin{cor}
If $S$ and $T$ are open trees then 
	\begin{equation*}
S \otimes \partial T \cup_{\partial S \otimes \partial T} \partial S\otimes  T \rightarrow S\otimes T
	\end{equation*}
is a monomorphism (hence, $S\otimes T$ being normal, a normal monomorphism).  
\end{cor}

Let us call a dendroidal set \emph{open} if it is a colimit of open trees (equivalently, a presheaf on the subcategory $\Omega^\circ \subseteq \Omega$ of open trees). Then by the usual induction, the previous corollary also gives 

\begin{cor}\label{cor:2}
Let $A \rightarrowtail X$ and $B \rightarrowtail Y$ be normal monomorphisms between open dendroidal sets. Then 
	\begin{equation*}
A\otimes Y \cup_{A \otimes B} X\otimes B \rightarrow X\otimes Y
	\end{equation*}
is again a normal monomorphism between open dendroidal sets. 
\end{cor}
 
\begin{proof}[Proof of Proposition \ref{prop}]
(i) Let $A\subseteq S\otimes \partial_y T$ and $B\subseteq S\otimes \partial_z T$ be shuffles, and suppose $F$ is a face of $A$ as well as of $B$. 
\begin{itemize}
\item In case $y$ is an edge, $F$ is a face of $S\otimes \partial_z T$ in which the colour $y$ does not occur. So by the shuffle lemma, $F$ is a face of a shuffle of $S\otimes D_y\partial_zT= S\otimes (\partial_y T\cap \partial_z T)$.

\item In case $y$ and $z$ are both vertices, suppose that at least one of them is not bald, say $y$. Let $y_1,\dots, y_n$ be all the external edges attached to this vertex $y$. (These are leaves, plus possibly the root.) Then $F$ is a face of $B$ in which none of these occur. So by the shuffle lemma again, $F$ is a face of a shuffle of $S\otimes D_{y_i}\partial_z T $ for each $i$. 
But the connected components of $D_{y_i}(\partial_z T)$ are the single edges $y_j$ ($j\neq i$) together with the face $\partial_y(\partial_z T)$. Since the $y_i$ do not occur in $F$, we find that $F$ is a face $S\otimes \partial_y \partial_z(T)$. 

\item Finally, in case $y$ and $z$ are both bald vertices, they also occur as bald vertices on top of the shuffles $B$ and $A$, respectively. So $F$ is a face of $B$ say, in which this bald vertex $y$ cannot occur (not as a single vertex, nor as a composition). But such a face is already a face of  a shuffle of $S\otimes \partial_y\partial_y(T)=S\otimes (\partial_yT \cap \partial_z T)$. 
\end{itemize}

(ii) Let $A$ be a shuffle of $\partial_x S \otimes T$ and $B$ a shuffle of $S\otimes \partial_y T$, and suppose $F$ is a face of $A$ as well as of $B$. Hence $F$ is a face of $S\otimes \partial_y T$ in which $x$ does not occur. By the shuffle lemma, $F$ is a face
of $\partial_xS \otimes \partial_y T$. 
\end{proof}

\end{document}